\documentclass[11pt,reqno]{amsart}
\usepackage{amssymb}
\usepackage{upgreek}
\usepackage{dsfont }
\usepackage{mathrsfs}
\usepackage{mathtools}
\usepackage[all]{xy}
\usepackage{color}
\usepackage{verbatim}
\usepackage{textcomp}
\usepackage{esint}
\usepackage[shortlabels]{enumitem}
\usepackage{hyperref}
\usepackage{url}
\allowdisplaybreaks

\newtheorem{thm}{Theorem}[section]
\newtheorem{lem}[thm]{Lemma}
\newtheorem{prop}[thm]{Proposition}

\theoremstyle{definition}

\theoremstyle{remark}

\newcommand{\bt}{\beta}

\newcommand{\ep}{\varepsilon}

\newcommand{\sm}{\sigma}

\newcommand{\om}{\omega}

\newcommand{\Hess}{\mathrm{Hess}}

\newcommand{\vol}{\mathrm{vol}}

\newcommand{\R}{{\mathbb{R}}}
\newcommand{\C}{{\mathbb{C}}}

\newcommand{\nab}{\nabla}

\newcommand{\ds}{\displaystyle}

\newcommand{\be}{\begin{equation}}
\newcommand{\ee}{\end{equation}}

\begin{document}

%\numberwithin{equation}{section}

\title[Canonical identification at infinity for Ricci-flat manifolds]
{Canonical identification at infinity for Ricci-flat manifolds}

\author[Jiewon Park]{Jiewon Park}

\maketitle

\begin{abstract}
We give a natural way to identify between two scales, potentially arbitrarily far apart, in a non-compact Ricci-flat manifold with Euclidean volume growth when a tangent cone at infinity has smooth cross section. The identification map is given as the gradient flow of a solution to an elliptic equation.
\end{abstract}

\vskip 1pc

\section{Introduction}

\vskip 1pc

Let $(M,g)$ be a complete, non-compact Riemannian manifold of dimension $n \geq 3$ with nonnegative Ricci curvature, with a fixed point $p \in M$. By Gromov's compactness theorem \cite{Gro} any sequence of rescalings $(M,r_i^{-2}g, p)$ with $r_i \to \infty$ has a subsequence that converges in the pointed Gromov-Hausdorff topology to a length space.   Any such limit is said to be a {\em tangent cone} at infinity of $M$. When $M$ has Euclidean volume growth, i.e., for all $r>0$ and some $v>0$,
\be
\vol (B_r(p))\geq v r^n,
\ee
then it is a result of Cheeger-Colding \cite{CC, CC2} that any tangent cone is a metric cone: it is isometric to a warped product $C(X)=[0,\infty) \times_r X$, where the cross section $X$ is a compact metric space. 

In general tangent cones may depend on the choice of rescalings $\{r_i\}$ and might not be unique; one might see different cones at different scales \cite{Pe, CC2, CN2, H1}. However, Colding-Minicozzi \cite{CM3} showed the following uniqueness theorem.

\begin{thm}[\cite{CM3}, Theorem 0.2] \label{CM_unique}
Let $M^n$ be a complete non-compact Ricci-flat manifold of Euclidean volume growth. If one tangent cone at infinity has smooth cross section, then the tangent cone is unique.
\end{thm}

Ricci-flat manifolds with Euclidean volume growth are important objects in numerous areas of mathematics and physics, including K\"ahler and Sasaki geometry, general relativity, and string theory; see for instance \cite{TY1, CH1, CH2, CW1, Sz1, Sz2, MSY2, MSY3, BG, VC1, K1} among others.

\vskip 1pc

In this paper we give a strengthening of Theorem \ref{CM_unique} by showing that there is an essentially canonical way of identifying any two scales even when they are very different. The identification itself is given as the gradient flow of a solution to an elliptic equation and thus, in particular, is a diffeomorphism.

\vskip 1pc

To define the identification map, recall that $M$ is {\it nonparabolic} by the work of Varopoulos \cite{V}, i.e., $M$ possesses a minimal positive symmetric Green function $G$ for the Laplacian. We will use the normalization $\Delta G(x,\cdot) = -n(n-2) \om_n \delta_x$ where $\om_n$ is the volume of the unit $n$-ball, so that $G = r^{2-n}$ when $(M,g,p) = (\R^n, g_{\mathrm{Euc}},0)$. We will denote the single-variable function $G(p, \cdot)$ simply as $G$. Define the function $b$ (cf. \cite{C, CM4, CM3}) by 
\be
b = G^{\frac{1}{2-n}},
\ee
and we denote the gradient flow of $b^2$ by $\Phi:M \times \R \to M$.

\vskip 1pc

As a motivation recall how $\Phi$ identifies different scales in the special case of the Euclidean space $(M,g,p) = (\R^n, g_{\mathrm{Euc}},0)$. In this case $b=r$ and $\Phi$ is simply a dilation map $\Phi_t(x) = e^{2t} x$. Thus $\Phi_t$ identifies two scales by the rescaling $\Phi_t^*g_{\mathrm{Euc}} = e^{4t}g_{\mathrm{Euc}}$. To state matters more naturally without the factor of $t$, we perform a coordinate change to the metric $g_{\mathrm{Euc}} = dr^2 + r^2 g_{\mathbb{S}^{n-1}}$ by $s = \log r$, so that $r^{-2} g_{\mathrm{Euc}} = ds^2 + g_{\mathbb{S}^{n-1}}$ is now a cylindrical metric. Since $r(\Phi_t(x)) = e^{2t}x$, it follows that 
\begin{equation}
(r \circ \Phi_t)^{-2} \Phi_t^* g_{\mathrm{Euc}} = g_{\mathrm{Euc}}, \nonumber
\end{equation}
so the metric $(r\circ \Phi_t)^{-2} \Phi_t^* g_{\mathrm{Euc}}$ is constant in $t$.

\vskip 1pc

Our theorem generalizes this example to Ricci-flat manifolds.  It identifies two scales on average after performing a conformal change to bring the metric in cylindrical form, and gives the rate of how fast they become similar. We will use the symbol $\fint$ to denote average integrals. So for instance, the notation $\ds{\fint_{b=r} f d\sm}$ refers to the average integral $\ds{\frac{1}{\mathcal{H}^{n-1}(\{b=r\})} \int_{b=r} f d\sm}$ over a level set $\{b=r\}$ where $d\sm$ is the area measure.

\vskip 1pc

\begin{thm} \label{mainthm2}
Let $(M^n,g)$ be a complete Ricci-flat manifold of dimension $n \geq 3$ and Euclidean volume growth. Define the family of metrics $g(t)$ by
\be
g(t) = (b\circ \Phi_t)^{-2} \Phi_t^* g.
\ee
Suppose that a tangent cone at infinity of $M$ has smooth cross section. Then there exist constants $C,r_0, \bt>0$ so that for any $r>r_0$ and $T>t>0$,
\begin{align} \label{mainthmest}
\fint_{b=r}\left\{ \sup_{v \neq 0} \left|\log \frac{g(T)(v,v)}{g(t)(v,v)} \right|\right\}  d\sm \leq C t^{-\frac{\bt}{2}}.
\end{align}
\end{thm}

\vskip 1pc

In equation (\ref{mainthmest}), the bound $Ct^{-\frac{\bt}{2}}$ is independent of $T$ and decreases with $t$. Thus the scale $\Phi_t(\{b=r\})$ is identified with the scale $\Phi_T(\{b=r\})$ for any $T>t$ and the estimate becomes better if $t$ is large. Also note that (\ref{mainthmest}) holds in particular for $r = \exp(At) r_0$ for a constant $A>0$, which is roughly the scale at time $t$.

\vskip 1pc

 The key ingredient in the proof of Theorem \ref{mainthm2} is a monotone non-decreasing quantity that approaches zero at a desirably fast rate. It is well-known that monotonicity formulae for elliptic and parabolic operators have a large number of geometric applications \cite{C, CM5, AM1, FMP1, AFM1} (see also \cite{CM6} for a survey). The rapid decay of this monotone quantity follows from an infinite dimensional \L ojasiewicz inequality \cite{CM3}. It is a critical component in proving the uniqueness of the tangent cone in Theorem \ref{CM_unique}.
 
Many examples are known of complete manifolds with nonnegative Ricci curvature and Euclidean volume growth that have non-isometric tangent cones \cite{Pe, CC2}, or even non-homeomorphic tangent cones \cite{CN2}. Even in the Ricci-flat case, Hattori \cite{H1} constructed a 4-dimensional manifold (with less than Euclidean volume growth) that has infinitely many non-isometric tangent cones at infinity. Other uniqueness results include an earlier result of Cheeger-Tian \cite{CT1}, where uniqueness was shown under the additional assumption of the integrability of cross sections and quadratic decay of sectional curvature. On the complex geometric side, Donaldson-Sun \cite{DS2} proved uniqueness of tangent cone at singularities of a limit of K\"ahler-Einstein manifolds, which heavily used the complex algebraic structure of the limit space. In fact, sometimes the algebraic geometry of the singularity can even be so powerful to determine the tangent cone \cite{HS1}.
 
Finally, we mention that Colding-Naber \cite{CN1} showed that even in the non-unique case, the tangent cones at a point vary in a H\"older continuous manner as one moves the point along a limit geodesic in the limit space. They also proved that the set of points that have unique tangent cone is convex.

\vskip 1pc

The rest of the paper is devoted to proving Theorem \ref{mainthm2}. In Section \ref{section2}, we will first relate the change in $t$ of the metric $g(t)$ to a geometric quantity on $M$. The quantity will be given as a weighted $L^2$-norm of the trace-free Hessian of $b^2$, i.e., $\ds{\left(\Hess_{b^2} - \frac{\Delta b^2}{n} g\right)}$. This quantity is useful since it decays fast enough to imply that the change in $t$ of $g(t)$ is small, as shown in Section \ref{section3}. In fact, $g(t)$ and the integral estimate in Theorem \ref{mainthm2} are designed precisely to bring this quantity into play.

\vskip 1pc

\section{Bounding $g(t)$ by the geometry of $M$} \label{section2}

\vskip 1pc

We remark that throughout this paper, the Laplacian is the trace of the Hessian with respect to $g$, and the Green function $G$ satisfies 
\begin{equation}
\Delta_x \int_M G(x,y) f(y) \, dy = -f(x), \hspace{1cm} \forall f\in\mathcal{C}_0^\infty(M).
\end{equation}
Since $G$ is harmonic away from the pole $p$, $b^2=G^{2/(2-n)}$ satisfies the equation
\begin{equation}
\Delta b^2 = 2n|\nab b|^2.
\end{equation}

We first compute the time derivative of $g(t)$. 

\vskip 1pc

\begin{lem} \label{g'(t)}
$g'(t)$ is given by
\begin{equation}
	\frac{d}{dt} g(t)= 2\Phi_t^*\left[b^{-2} \left(\Hess_{b^2} -\frac{\Delta b^2}{n} g\right)\right].
\end{equation}
Here the Hessian and the Laplacian are taken with respect to the original metric $g$.
\end{lem}

\vskip 1pc

\begin{proof} Fix any point $x \in M \backslash \{p\}$, and a tangent vector $v \in T_x M$. Then we have the following calculation, where $\mathcal{L}$ is the usual Lie derivative.

\begin{align*}
& \hspace{-0.8cm} \left(\frac{d}{dt}(b \circ \Phi_t)^{-2} \Phi_t^*g\right)_x (v,v) \\
&= (\mathcal{L}_{\nab b^2} b^{-2})(\Phi_t(x))\cdot (\Phi_t^* g)_{x} ( v, v)\\
&\hspace{1cm}+ b\left(\Phi_t(x)\right)^{-2} \cdot (\mathcal{L}_{\nab b^2} g)_{\Phi_t(x)} ((d\Phi_t)_{x} v,(d\Phi_t)_{x}v)\\
&=g(\nab b^2, \nab b^{-2})(\Phi_t(x)) \cdot (\Phi_t^*g)_{x}(v,v)\\
&\hspace{1cm}+ b\left(\Phi_t(x)\right)^{-2} \cdot 2(\Hess_{b^2})_{\Phi_t(x)} ((d\Phi_t)_{x} v,(d\Phi_t)_{x} v)\\
&=-4b(\Phi_t(x))^{-2}|\nab b|^2(\Phi_t(x)) \cdot (\Phi_t^*g)_{x}(v,v)\\
&\hspace{1cm}+ b\left(\Phi_t(x)\right)^{-2} \cdot 2(\Hess_{b^2})_{\Phi_t(x)} ((d\Phi_t)_{x} v,(d\Phi_t)_{x} v) \\
&=2b(\Phi_t(x))^{-2}\left(\Phi_t^*\left(\Hess_{b^2}-2|\nab b|^2 g\right)\right)_{x}(v,v) \\
&=2b(\Phi_t(x))^{-2}\left(\Phi_t^*\left(\Hess_{b^2}-\frac{\Delta b^2}{n} g\right)\right)_{x}(v,v).
\end{align*}
This proves the lemma.
\end{proof}

\vskip 1pc

\begin{lem} \label{g'(t)norm}
The norm of the time derivative of $g(t)$ is controlled by the Hessian of $b^2$:
\begin{equation}
	\left\|\frac{d}{dt} \, g(t)\right\|^2_{g(t)}(x) \leq 4\left\|\Hess_{b^2}^g - \frac{\Delta b^2}{n} g\right\|^2_g(\Phi_t(x)).
\end{equation}	
\end{lem}

\vskip 1pc

\begin{proof}
Let $v \in T_x M$, $v \neq 0$. Then using Lemma \ref{g'(t)} we compute
\begin{align*}
\frac{g'(t)(v,v)_x}{g(t)(v,v)_x} &= \frac{2\Phi_t^*\left(b^{-2} \left(\Hess_{b^2} -\frac{\Delta b^2}{n} g\right)\right)_{x}(v,v)}{(b\circ \Phi_t(x))^{-2}(\Phi_t^*g)_{x}(v,v)} \\
&=\frac{2\Phi_t^*\left(\Hess_{b^2} -\frac{\Delta b^2}{n} g\right)_{x}(v,v)}{(\Phi_t^*g)_{x}(v,v)} \\
&=2\left(\Hess_{b^2} -\frac{\Delta b^2}{n} g\right)_{\Phi_t(x)} \left(\frac{d(\Phi_t)_{x}v}{|d(\Phi_t)_{x} v|_g}, \frac{d(\Phi_t)_{x} v}{|d(\Phi_t)_{x} v|_g} \right).
\end{align*}
Since $\ds{\left|  \frac{d(\Phi_t)_{x} v}{|d(\Phi_t)_{x} v|_g}  \right|_g = 1}$, we have that
\begin{align*}
\left| \frac{g'(t)_{x}(v,v)}{g(t)_{x}(v,v)} \right|^2 &\leq 4 \left\{ \sup_{\underset{|w|=1}{w \in T_{\Phi_t(x)} M}}\left| \left(\Hess_{b^2} -\frac{\Delta b^2}{n} g\right)(w,w) \right| \right\}^2 \\
&\leq 4\left\|\Hess_{b^2} -\frac{\Delta b^2}{n} g\right\|^2(\Phi_t(x)).
\end{align*}
This proves the lemma.
\end{proof}

\vskip 1pc

\begin{lem} \label{prop1}
For any $0<s<t$, we have
\begin{align}
\fint_{b=r}\left\{ \sup_{v \neq 0}  \left|\log \frac{g(t)(v,v)}{g(s)(v,v)} \right| \right\}d\sm \leq \sqrt{t-s}\left(\fint_{b=r} \int_s^t \|g'(\tau)\|^2_{g(\tau)} d\tau \, d\sm\right)^{1/2}.
\end{align}
\end{lem}

\vskip 1pc

\begin{proof}
First note that, for any $x\in \{b=r\}$ and $v \in T_x M$, $v \neq 0$, we have
\begin{equation*}
\left(\log \frac{g(t)(v,v)}{g(s)(v,v)}\right)(x)= \int_s^t \frac{g'(\tau)_{x}(v,v)}{g(\tau)_{x}(v,v)} \, d\tau.
\end{equation*}
Therefore we have
\begin{align*}
\fint_{b=r} \sup_{v \neq 0} \left|\log \frac{g(t)(v,v)}{g(s)(v,v)} \right| d\sm &\leq \fint_{b=r} \sup_{v\neq 0} \int_s^t \left| \frac{g'(\tau)(v,v)}{g(\tau)(v,v)}\right| \, d\tau \, d\sm\\
&=\frac{1}{\mathcal{H}^{n-1} (\{b=r\})} \int_{b=r} \int_s^t  \left( \|g'(\tau)\|_{g(\tau)}^{2} \right)^{\frac{1}{2}} \, d\tau \, d\sm \\
&\leq \frac{\sqrt{t-s}}{\mathcal{H}^{n-1} (\{b=r\})^{1/2}}\left( \int_{b=r} \int_s^t  \|g'(\tau)\|_{g(\tau)}^2 \, d\tau \, d\sm\right)^{1/2},
\end{align*}
where we used the Cauchy-Schwarz inequality for the last step. This proves the lemma.
\end{proof}

\vskip 1pc

In the next proposition we will bound the $L^2$-norm of $g'(t)$ by a weighted $L^2$-norm of the
trace-free Hessian of $b^2$.

\vskip 1pc

\begin{prop} \label{L2-of-g'(t)}
There exists $r_0 = r_0(M,g)$ such that, for all $r>r_0$, the following is true. Let $F : \{b>r_0\} \to M$ be the map that sends a point in $\{b>r_0\}$ to the unique point in the same flow line that belongs to $\{b=r\}$. Then we have that
\begin{align} \label{intermediate}
\fint_{b=r} \int_s^t \|g'(\tau)\|^2_{g(\tau)} \, d\tau \, d\sigma \nonumber \\
	& \hspace{-4cm} \leq  \frac{2r^{n-1}}{\mathcal{H}^{n-1} (\{b=r\})} \int_{\underset{s \leq \tau \leq t}{\cup} \Phi_\tau(\{b=r\})} \frac{b^{-n}}{|\nab b| \circ F} \left\|\Hess_{b^2}^g - \frac{\Delta b^2}{n} g\right\|^2.
\end{align}	
\end{prop}

\vskip 1pc

\begin{proof}
We will first derive inequality (\ref{intermediate}) assuming that the set $\{b \geq r_0\}$ does not contain any critical point of $b$ for large $r_0$. This assumption will be removed at the end of the proof. Let $r>r_0$. We are going to consider $\Phi: \{b=r\} \times (s,t) \to M $, $\Phi(x,\tau ) = \Phi_\tau (x)$ to be a parametrization of the open set $\underset{s < \tau < t}{\cup} \Phi_\tau(\{b=r\})$, and $d\tau \, d\sm$ as a top-degree form on this open set.  Note that $d\sm$ is the $(n-1)$-dimensional volume form on $\{b=r\}$ with respect to $g$, so that if $x_1, \cdots, x_{n-1}$ are coordinates on $\{b=r\}$ then $\ds{d\sm = \sqrt{\det g|_{\{b=r\}}} dx_1 \cdots dx_{n-1}}$. In particular, $\ds{d\tau \, d\sm = \sqrt{\det g|_{\{b=r\}}} d\tau\, dx_1 \cdots dx_{n-1}}$. So if $x \in \{b=r\}$, then $d\tau \, d\sm_x$ can define a top-degree form at $\Phi_{\tau} (x)$ for any value of $\tau$.

Let $x$ denote a point in $\{b=r\}$. Then by Lemma \ref{g'(t)norm}, we have 
\begin{align} \label{average}
\fint_{b=r} \int_s^t \|g'(\tau)\|^2_{g(\tau)}(x) \, d\tau \, d\sigma_{x} \nonumber \\
	&\hspace{-5cm} \leq \frac{4}{\mathcal{H}^{n-1} (\{b=r\})} \int_{b=r} \int_s^t \left\|\Hess_{b^2}^g - \frac{\Delta b^2}{n} g\right\|^2_g(\Phi_{\tau}(x)) \, d\tau \, d\sigma_x  \nonumber \\
& \hspace{-5cm}=\frac{4}{\mathcal{H}^{n-1} (\{b=r\})} \int_{\underset{s < \tau < t}{\cup} \Phi_\tau(\{b=r\})} \left\|\Hess_{b^2}^g - \frac{\Delta b^2}{n} g\right\|^2_g(\Phi_{\tau}(x)) \, \frac{d\tau \, d\sigma_x}{d\vol_{\Phi_{\tau}(x)}} \,d\vol_{\Phi_{\tau}(x)}.
\end{align}

We have to compare the form $d\tau\, d\sm_x$ to the actual volume form $d\vol_{\Phi_\tau(x)}$ at $\Phi_{\tau}(x)$, $\tau \in (s,t)$. Recall that the Laplacian is the change of the volume element along the flow line, i.e.,
\begin{equation}
\frac{d}{d\tau} \left( \frac{\Phi^*_\tau (d\vol)}{d\vol}\right)_x=\Delta(b^2)(\Phi_\tau(x)) \left( \frac{\Phi^*_\tau (d\vol)}{d\vol}\right)_x.
\end{equation}
By the previous observation we can interpret $d\tau \, d\sm_x$ as a top-degree form either at $\Phi_{\tau}(x)$ or at $x$, and calculate that

\begin{align} \label{vol_change}
\frac{d\tau d\sm_x}{d\vol_{\Phi_\tau(x)}} &= \frac{d\tau \, d\sm_x}{d\vol_x} \frac{d\vol_x}{d\vol_{\Phi_\tau(x)}} \nonumber \\
&=\frac{d\tau \, d\sm_x}{\frac{1}{|\nab b^2|}\, db^2 \, d\sm_x}\cdot \exp\left(-\int_0^\tau \Delta b^2 (\Phi_u(x)) \, du \right) \nonumber \\
&=|\nab b^2|(x) \cdot \frac{1}{\mathcal{L}_{\nab b^2} b^2} \cdot  \exp\left(-\int_0^\tau 2n|\nab b|^2 (\Phi_u(x)) \, du \right) \nonumber \\
&=|\nab b^2|(x) \cdot \frac{1}{|\nab b^2|^2(x)} \cdot \exp\left(-\int_0^\tau 2n|\nab b|^2 (\Phi_u(x)) \, du \right) \nonumber \\
&=\frac{1}{|\nab b^2|(x)} \cdot \exp\left(-\int_0^\tau 2n|\nab b|^2 (\Phi_u(x)) \, du \right).
\end{align}

On the other hand, note that the change of $\log b$ along a flow line is given by
\begin{align}
\frac{d}{d\tau} \log b = \mathcal{L}_{\nab b^2} \log b= g\left(\nab b^2, \frac{\nab b}{b}\right) = 2|\nab b|^2.
\end{align}
Hence, the change of $b$ along the flow line is 
\begin{equation} \label{change_of_b}
b(\Phi_\tau(x)) = b(x) \cdot \exp \left(\int_0^\tau 2|\nab b|^2(\Phi_u(x)) \, du \right).
\end{equation}
Combining with (\ref{vol_change}), we have that 
\begin{align} \label{volc}
\frac{d\tau d\sm_x}{d\vol_{\Phi_\tau(x)}} &=\frac{1}{|\nab b^2|(x)} \cdot \left(\frac{b(\Phi_\tau(x))}{b(x)}\right)^{-n}=\frac{r^{n-1}}{2|\nab b|(x)}b(\Phi_\tau(x))^{-n}.
\end{align}
Substituting (\ref{volc}) into (\ref{average}) yields
\begin{align}
\fint_{b=r} \int_s^t \|g'(\tau)\|^2_{g(\tau)}(x) \, d\tau \, d\sigma_{x} \nonumber \\
&\hspace{-4cm} \leq \frac{2r^{n-1}}{\mathcal{H}^{n-1} (\{b=r\})} \int_{\underset{s \leq \tau \leq t}{\cup} \Phi_\tau(\{b=r\})} \frac{1}{|\nab b|(x)} b^{-n}\left\|\Hess_{b^2}^g - \frac{\Delta b^2}{n} g\right\|^2,
\end{align}
which is the inequality that we wanted to prove.

Finally we argue as promised that if $r$ is large then $\{b \geq r\}$ does not contain any critical point of $b$. This argument is contained in \cite{CM3} but we repeat it for the convenience of the reader. We denote by $\Psi(r)$ a positive function of one variable such that $\Psi(r) \to 0$ as $r\to\infty$. $\Psi(r)$ may change from line to line.

 Let $C(N)$ be a smooth tangent cone of $M$, which is a metric cone by \cite{CC2}. Denote by $o$ the vertex of $C(N)$. Let $A>1$ be a fixed constant. Then there exists a sequence $r_i \to \infty$ such that
\be
d_{GH} \left( B_p(Ar_i) , B^{C(N)}_o(Ar_i)  \right) < r_i^{2} \Psi(r_i),
\ee
where $\Psi(r) \to 0$ as $r\to \infty$. By \cite{C2}, this convergence in the Gromov-Hausdorff topology is in fact a convergence in the $\mathcal{C}^\infty$ topology since $M$ is Einstein.

On the other hand, by the Bishop-Gromov volume comparison theorem, the rate of volume growth $\ds{\frac{\vol (B_p(r))}{\om_n r^{n}}}$ is monotone non-increasing. Let $\ds{V_M = \lim_{r \to \infty} \frac{\vol (B_p(r))}{\om_n r^{n}}}$ and define a constant $b_\infty$ by
\be 
b_\infty = \left(\frac{V_M}{\om_n}\right)^{\frac{1}{n-2}}.
\ee

In \cite{CM4} the following integral gradient estimate of $b$ was obtained,

\be
\fint_{B_p(Ar) } \left| |\nab b|^2-b_\infty^2 \right|^2 \leq \Psi(r).
\ee
Since $b$ satisfies an elliptic equation, the integral gradient estimate implies pointwise gradient bounds
\begin{equation} \label{grad_bound_seq}
\sup_{B_p(Ar_i) \backslash B_p(r_i / A)} \left||\nab b|-b_\infty\right| \leq \Psi(r_i).
\end{equation}
In particular, if $\ds{\frac{r_i}{A} \leq b}$ for large $r_i$ then $|\nab b| \neq 0$. This completes the proof.
\end{proof}

\vskip 1pc

Combining Lemma \ref{prop1} and Proposition \ref{L2-of-g'(t)} gives the following proposition.

\vskip 1pc

\begin{prop} \label{bound-by-hess}
If $r>r_0$, then for any $0<s<t$,

\begin{align}
\fint_{b=r} \sup_{v \neq 0} \left|\log \frac{g(t)(v,v)}{g(s)(v,v)} \right| d\sm_{x} \nonumber \\
&\hspace{-4cm}  \leq \sqrt{t-s}\left( \frac{2r^{n-1}}{\mathcal{H}^{n-1} (\{b=r\})} \int_{\underset{s \leq \tau \leq t}{\cup} \Phi_\tau(\{b=r\})} \frac{1}{|\nab b|(x)} b^{-n}\left\|\Hess_{b^2}^g - \frac{\Delta b^2}{n} g\right\|^2\right)^{1/2}.
\end{align}
\end{prop}

\vskip 1pc

\section{Proof of Theorem \ref{mainthm2}} \label{section3}

\vskip 1pc

We begin this section with a lemma on controlling the area of level sets of $b$.

\vskip 1pc

\begin{lem} \label{area}
Suppose that $M$ has a smooth tangent cone $C(Y)$ at infinity. Then for any $\ep>0$, if $r$ is sufficiently large, then 
\begin{equation} \label{gradest_p}
\left| |\nab b| - b_\infty \right| < \ep \, \, \mathrm{when} \, \, r \leq b.
\end{equation}
Moreover, there is a constant $C=C(M,g)>0$ such that the area of the hypersurface $\{b=r\}$ is bounded above and below,
\begin{equation}
\frac{r^{n-1}}{C} \leq \mathrm{Area}(\{b=r\}) \leq Cr^{n-1}.
\end{equation}
\end{lem}

\vskip 1pc

\begin{proof}
The first assertion was already proved in \cite{CM3}, as we mentioned in Section 2 (see equation (\ref{grad_bound_seq})). For the second assertion we will utilize the following facts from \cite{C} (Corollary 2.19, Theorem 2.12): define a function $A(r)$ by
\begin{equation} \label{def_A}
A(r) = r^{1-n}\int_{b=r} |\nab b|^3 d\sm.
\end{equation}
Then $A$ is monotone non-increasing in $r$. Moreover, 
\begin{equation}\ds{\lim_{r \to 0} A(r)<\infty},
\end{equation} 
and 
\begin{equation}
\ds{0<\lim_{r\to \infty} A(r)}.
\end{equation} 
Now by (\ref{gradest_p}), if $r$ is large then there exists $C>0$ so that 
\begin{equation} \label{A_and_area}
\frac{A(r)}{C} \leq \mathrm{Area}(\{b=r\}) \leq CA(r).
\end{equation}
The second assertion in the lemma now follows from (\ref{def_A})--(\ref{A_and_area}) and (\ref{gradest_p}).
\end{proof}

\vskip 1pc

We will make use of the following (corollary of the) \L ojasiewicz inequality of Colding-Minicozzi to prove Theorem \ref{mainthm2}.

\vskip 1pc

\begin{thm}[\cite{CM3}, Proposition 2.25] \label{loj}
Suppose that one tangent cone $C(N)$ at infinity of $M$ is smooth. Then there is a constant $C=C(\ep)>0$ such that the following is true: if $d_{GH}(B_{2R}(p) \backslash B_R(p), B^{C(N)}_{2R}(0) \backslash B^{C(N)}_R(0)) \leq \ep R$ for all $R\in [\frac{s}{100}, 100r]$, then
\begin{equation}
\int_{b\geq r} 	b^{-n} \left\|\Hess_{b^2}^g - \frac{\Delta b^2}{n}g\right\|^2_g  	\leq \frac{C}{\log(r/s)^{1+\bt}}.
\end{equation}
Here, $\bt>0$ is a constant depending on $(M^n,g)$, but not on $\ep$.
\end{thm}

\vskip 1pc

We will also utilize the following comparison result for $G$ from \cite{MSY}.

\vskip 1pc

\begin{thm}[\cite{MSY}, Subsection 1.2] \label{MSY}
Let $N$ be a complete Riemannian manifold with Euclidean volume growth and nonnegative Ricci curvature of dimension $n \geq 3$. Then there exist constants $C_1, C_2$ with the following effect. If $G$ is the minimal positive Green function with pole $p \in N$ and $r$ is the distance from $p$, then
\begin{equation}
C_1 r^{2-n} \leq G \leq C_2 r^{2-n}.
\end{equation}
\end{thm}

\vskip 1pc

Now we can finish the proof of Theorem \ref{mainthm2}.

\vskip 1pc

\begin{proof}[Proof of Theorem \ref{mainthm2}]
Fix a small number $\ep>0$. Since $C(N)$ is the unique tangent cone by Theorem \ref{CM_unique}, there exists $r_0$ so that if $r>r_0$ then
\begin{equation}
\left||\nab b| - b_\infty \right| < \ep \, \, \mathrm{when} \, \, r \leq b,
\end{equation}
and
\begin{equation}
	d_{GH}(B_{2R}(p) \backslash B_R(p), B^{C(N)}_{2R}(0) \backslash B^{C(N)}_R(0)) \leq \ep R \, \, \mathrm{for \, all} \, \, R \geq \frac{r}{100}. 
\end{equation}
Recall the following inequality from Proposition \ref{bound-by-hess}, that for $0<s<t$,

\begin{align}
\fint_{b=r} \sup_{v \neq 0} \left|\log \frac{g(t)(v,v)}{g(s)(v,v)} \right| d\sm \nonumber \\
&\hspace{-4cm} \leq \sqrt{t-s}\left( \frac{2r^{n-1}}{\mathcal{H}^{n-1} (\{b=r\})} \int_{\underset{s \leq \tau \leq t}{\cup} \Phi_\tau(\{b=r\})} \frac{1}{|\nab b|(x)} b^{-n}\left\|\Hess_{b^2}^g - \frac{\Delta b^2}{n} g\right\|^2\right)^{1/2}.
\end{align}
Since $(b_\infty - \ep) \leq |\nab b|$, it follows that there exists a positive constant $C$ depending only on $b_\infty$ so that

\begin{align}
\fint_{b=r} \sup_{v \neq 0} \left|\log \frac{g(t)(v,v)}{g(s)(v,v)} \right| d\sm  \nonumber \\
&\hspace{-3cm} \leq C \sqrt{t-s}\left( \frac{r^{n-1}}{\mathcal{H}^{n-1} (\{b=r\})} \int_{\underset{s \leq \tau \leq t}{\cup} \Phi_\tau(\{b=r\})}b^{-n}\left\|\Hess_{b^2}^g - \frac{\Delta b^2}{n} g\right\|^2\right)^{1/2}.
\end{align}

Next we look at the region of integration on the right hand side. From this point $C$ is allowed to change line by line, as long as it is independent of $r$, $s$, and $t$. Recall that by Lemma \ref{area}, we have
\begin{align}
\mathrm{Area}(\{b=r\}) \geq Cr^{n-1}.
\end{align}
Also note that for $x\in \{b=r\}$, by equation (\ref{change_of_b}), 
\begin{align}
	b(\Phi_\tau(x))&=b(x) \exp\left(\int_0^\tau 2 |\nab b|^2(\Phi_s(x)) \, ds\right) \nonumber  \\
	&\in \left[r\cdot\exp\left(2\tau (b_\infty -\ep)^2 \right), \, r\cdot\exp\left(2\tau (b_\infty+\ep)^2 \right)\right].
\end{align}
It follows that
\begin{align}
\underset{s \leq \tau \leq t}{\cup} \Phi_\tau(\{b=r\}) \subset \left\{r\cdot\exp\left(2s(b_\infty-\ep)^2 \right) \leq b \leq r\cdot\exp\left(2t(b_\infty+\ep)^2 \right)  \right\}.
\end{align}

Therefore, we have that
\begin{align}
\fint_{b=r} \sup_{v \neq 0} \left|\log \frac{g(t)(v,v)}{g(s)(v,v)} \right| d\sm \nonumber \\
&\hspace{-4cm} \leq C\sqrt{t-s}\left( \int_{re^{2s(b_\infty-\ep)^2} \leq b \leq r e^{2t(b_\infty+\ep)^2} } b^{-n}\left\|\Hess_{b^2}^g - \frac{\Delta b^2}{n} g\right\|^2\right)^{1/2} \nonumber \\
&\hspace{-4cm}\leq C\sqrt{t-s}\left( \int_{re^{2s(b_\infty-\ep)^2}\leq b} b^{-n}\left\|\Hess_{b^2}^g - \frac{\Delta b^2}{n} g\right\|^2\right)^{1/2}.
\end{align}
Combining with Theorem \ref{loj}, it follows that
\begin{align} \label{st}
\fint_{b=r} \sup_{v \neq 0} \left|\log \frac{g(t)(v,v)}{g(s)(v,v)} \right| d\sm \leq C\sqrt{t-s}\left(  \frac{1}{2s(b_\infty-\ep)^2+\log(r/r_0)}\right)^{\frac{1+\bt}{2}}.
\end{align}
Taking $t=As$ with $A>1$ in (\ref{st}) then switching $s$ and $t$, we obtain
\begin{align} 
\fint_{b=r} \sup_{v \neq 0} \left|\log \frac{g(At)(v,v)}{g(t)(v,v)} \right| d\sm
& \leq C\sqrt{(A-1)t} \cdot \left(  \frac{1}{2t(b_\infty-\ep)^2+\log(r/r_0)}\right)^{\frac{1+\bt}{2}} \nonumber \\
& \leq C\sqrt{(A-1)t} \cdot \left(  \frac{1}{2t(b_\infty-\ep)^2}\right)^{\frac{1+\bt}{2}} \nonumber \\
& \leq C\sqrt{A-1} \cdot t^{-\frac{\bt}{2}}.
\end{align}
Iterating for $t, At, A^2t, \cdots$, we have
\begin{align} 
\fint_{b=r} \sup_{v \neq 0} \left|\log \frac{g(A^n t)(v,v)}{g(t)(v,v)} \right| d\sm \leq C\sqrt{A-1}\frac{1-A^{-\frac{n\bt}{2}}}{1-A^{-\frac{\bt}{2}}} \cdot t^{-\frac{\bt}{2}} \leq \frac{C \sqrt{A-1}}{1-A^{-\frac{\bt}{2}}} t^{-\frac{\bt}{2}}.
\end{align}
Since this is true for any $n>0$ and $A>1$, we conclude that for any $T>t$, 
\begin{align}
\fint_{b=r} \sup_{v \neq 0} \left|\log \frac{g(T)(v,v)}{g(t)(v,v)} \right| d\sm \leq C t^{-\frac{\bt}{2}}.
\end{align}
 This finishes the proof of the theorem.
\end{proof}

\vspace{0.5cm}

{\bf Acknowledgement}. The author would like to thank Professor Tobias Holck Colding for numerous helpful discussions. The author was partially supported by NSF Grant DMS-1812142.

\vskip 1pc 

\noindent Department of Mathematics

\noindent Massachusetts Institute of Technology

\noindent 182 Memorial Drive, Cambridge, MA 02139

\noindent E-mail address: \url{jiewon@mit.edu}

\end{document}